\documentclass[12pt]{article}

\usepackage{amsmath}   
\usepackage{amsthm}
\usepackage{mathptmx}

\newtheorem{theorem}{Theorem}
\newtheorem{lemma}{Lemma}
\newtheorem*{lemma*}{Lemma}
\newtheorem{corollary}{Corollary}

\usepackage[nooneline]{caption}
\usepackage{pifont,hyperref,amsthm,amsmath,amssymb,color,multirow,graphicx,subfigure,color}
\numberwithin{equation}{section}

\RequirePackage{palatino}

\begin{document}


\begin{center}
{\Large\bf On asymptotic structure of continuous-time
    Markov Branching Processes allowing Immigration and without high-order moments}
\end{center}
\vspace{.1cm}
\begin{center}
{\sc Azam~A.~Imomov \, and \, Abror~Kh.~Meyliev}\\
\vspace{.4cm}
{\small \it Karshi State University, 17 Kuchabag street, \\
100180 Karshi city, Uzbekistan\\} 
{\small e-mail: \; {\small \sl imomov{\_}\,azam@mail.ru , \, abror{\_}meyliyev@mail.ru} }
\end{center}

\vspace{.3cm}
\begin{abstract}
    We observe the continuous-time Markov Branching Process without high-order moments and
    allowing Immigration. Limit properties of transition functions and their convergence to invariant measures
    are investigated. Main mathematical tool is regularly varying generating functions with remainder.


\end{abstract}


\medskip

\section {Introduction and main results}    \label{MySec:1}

    \subsection {Background and Basic assumptions}       \label{MySubsec:1.1}

    We deal with the model of population growth called continuous-time Markov Branching Process allowing Immigration (MBPI).
    This process can have a simple physical interpretation: a population size changes not only as a result of reproduction
    and disappearance of existing individuals, but also at the random stream of inbound ``extraneous'' individuals
    of the same type from outside.
    The population of individuals develops as follows. Each individual existing at time
    $t \in {\mathcal{T}}=[0, + \infty)$ independently of his history and of each other for a small time interval
    $(t, t+ \varepsilon)$ transforms into $j \in {\mathbb{N}}_0 \backslash \{ 1\} $ individuals with
    probability $a_j \varepsilon  + o(\varepsilon )$ and, with probability $1+ a_1 \varepsilon + o(\varepsilon)$
    stays to live or makes evenly one descendant (as $\varepsilon \downarrow 0$); where
    $\mathbb{N}_0=\{0\}\cup\mathbb{N}$ and $\mathbb{N}$ is the set of natural numbers. Here $\left\{{a_j} \right\}$
    are intensities of individuals' transformation that $a_j \ge 0$ for $j \in {\mathbb{N}}_0 \backslash \{1\}$ and
    $0 < a_0 < -a_1 = \sum\nolimits_{j \in {\mathbb{N}}_0 \backslash \{1\}}{a_j}<\infty$. Independently of these
    for each time interval $j \in {\mathbb{N}}$ new individuals inter the population with probability
    $b_j \varepsilon + o(\varepsilon)$ and, immigration does not occur with probability $1 + b_0 \varepsilon + o(\varepsilon)$.
    Immigration intensities $b_j \ge 0$ for $j \in {\mathbb{N}}$ and $0 <-b_0 = \sum\nolimits_{j \in {\mathbb{N}}}{b_j}<\infty $.
    Newly arrived individuals undergo transformation in accordance with the reproduction law generated
    by intensities $\left\{ {a_j} \right\}$; see {\cite[p.~217]{Sevast71}}.
    Thus, the process under consideration is completely determined by infinitesimal generating functions(GFs)
\begin{equation*}
    f(s) = \sum\limits_{j \in {\mathbb{N}}_0} {a_j s^j}
    \quad  \mbox{and} \quad  g(s)=\sum\limits_{j \in {\mathbb{N}}_0}{b_j s^j}
    \quad \parbox{2.4cm}{for {} $s\in{[0, 1)}$.}
\end{equation*}

    Let $X(t)$ be the population size at the time $t \in {\mathcal{T}}$ in MBPI. This is homogenous
    continuous-time Markov chain with state space $\mathcal{S}\subset\mathbb{N}_0$ and transition functions
\begin{equation*}
    p_{ij} (t):= \mathbb{P}_i \left\{ {X(t) = j} \right\}
    = \mathbb{P}\left\{{X(t+\tau)=j \, \bigl| \, {X(\tau)=i} \bigr.} \right\}
\end{equation*}
    for all $i,j \in {\mathcal{S}}$ and $\tau , t \in {\mathcal{T}}$. An appropriate probability GF
\begin{equation}                    \label{1.1}
    {\mathcal{P}}_i (t;s) = \sum\limits_{j \in {\mathcal{S}}} {p_{ij} (t)s^j}
    = \left(F(t;s)\right)^{i} \exp \left\{ {\int\limits_0^t {g\left({F(u;s)} \right)du}} \right\},
\end{equation}
    where $F(t;s)$ is GF of $Z(t)$ -- Markov Branching Process (MBP) initiated by single particle without immigration components.

    Providing that $m:=f'(1-)<\infty$, the value $1+m\varepsilon + o(\varepsilon)$ denotes the mean per capita offspring
    number of single individual during the any small time-interval $(t, t + \varepsilon)$ as $\varepsilon \downarrow 0$;
    and similarly, the value $\alpha \varepsilon + o(\varepsilon )$ denotes the average number of immigrants entering the
    population in this time interval, where $\alpha:=g'(1-)<\infty$.
    The case $\alpha =0$ specifies the process without immigration since then $g(s) \equiv 0$.

    Classification of states $\mathcal{S}$ depends on a value of the parameter $m$.
    According to the general classification of continuous-time Markov chains, the process
    is called subcritical, critical, and supercritical if $m<0$, $m=0$ and $m>0$, respectively.

    We consider the critical case only. In this case Sevastyanov~{\cite{Sevast57}} proved that
    if $2b:=f''(1-)$ is finite and the immigration law has a finite mean then the normalized process
    ${{X(t)}\mathord{\left/{\vphantom {{X(t)} {bt}}} \right. \kern-\nulldelimiterspace}{bt}}$
    has a limiting Gamma distribution function $\Gamma _{1,\,\lambda } (x)$, $x \ge 0$, where
    $\lambda = {{\alpha}\mathord{\left/{\vphantom {{\alpha}{b}}}\right.\kern-\nulldelimiterspace}{b}}$.
    Pakes~{\cite{PakesSankh}} has proved that $t^\lambda {\mathcal{P}}_i (t;s)$ converges as $t\to{\infty}$
    to the limit $\pi{(s)}$ which has the power series expansion and generates an invariant measure
    $\left\{{\pi}_j, j\in{\mathcal{S}} \right\}$ for MBPI iff $\sum\nolimits_{j \in {\mathbb{N}}}{a_j j^2 \ln{j}}<\infty$
    and $\sum\nolimits_{j \in {\mathbb{N}}}{b_j j\ln{j}}<\infty$. In accordance with the appropriate result
    of the paper {\cite{LiPakes2012}} the invariant measure of MBPI can also be constructed by the strong ratio
    limit property of transition functions but slightly different in appearance. Namely, the set of positive numbers
    $\bigl\{\upsilon_j :=\lim _{t \to \infty } {{p_{0j}(t)} / {p_{00} (t)}}\bigr\}$ is an invariant measure.
    It can be seen a close relation between the sequences $\left\{{\pi}_j, j\in{\mathcal{S}} \right\}$ and
    $\left\{\upsilon_j, j\in{\mathcal{S}}\right\}$, and their GFs $\pi{(s)}$ and
    $\mathcal{U}(s)= \sum\nolimits_{j \in{\mathcal{S}}}{\upsilon_j s^j}$. In fact,
    they are really only different versions of the same limit law. So, it is easy
    to see $\mathcal{U}(s)=\pi{(s)}/\pi{(0)}$, and this is consistent with uniqueness,
    up to a multiplicative constant, of the invariant measure of MBPI.

    In the circle of tasks of studying the asymptotic properties of process states, of exceptional
    interest is the estimate of the rate of convergence to invariant measures. In the paper
    {\cite{ImomovSFU14}}, under the condition $f'''(1-)<\infty$, the rate of convergence of
    $t^\lambda {\mathcal{P}}_i (t;s)$ to $\pi{(s)}$ was studied. It was found there that the convergence
    rate is $\mathcal{O}\bigl({{{\ln t}/ t}} \bigr)$ as $t\to{\infty}$ uniformly in $s\in [0, 1)$.

    In this report we attempt to improve the above results from {\cite{ImomovSFU14}} and {\cite{LiPakes2012}} on
    convergence rate and find out an appearance of GF of the invariant measure bypassing the finiteness conditions
    of high-order moments of infinitesimal GF $f(s)$ and $g(s)$. For this, we will substantially use elements of the
    theory of regularly varying functions in the sense of Karamata; see for instance {\cite{Bingham}} and {\cite{SenetaRV}}.

    Throughout the paper, we adhere to the following assumptions on $f(s)$ and $g(s)$:
\begin{equation*}
    f(s)=(1-s)^{1+\nu}\mathcal{L}\left({{{1} \over {1-s}}}\right)      \eqno[\textsf {$f_\nu$}]
\end{equation*}
    and
\begin{equation*}
    g(s)=-(1-s)^{\delta}{\ell}\left({{{1} \over {1-s}}}\right)      \eqno[\textsf {$g_\delta$}]
\end{equation*}
    for all $s\in [0, 1)$, where $0 < \nu , \delta < 1$ and functions $\mathcal{L}(\cdot)$, $\ell(\cdot)$ are
    \textit{slowly varying at infinity} (${\textbf{SV}}_\infty$). By the criticality of our process, the assumption
    $[f_\nu]$ implies that $2b:=f''(1-)=\infty$. If $b<\infty$ then $[f_\nu]$ holds with $\nu =1$ and
    $\mathcal{L}(t) \to b$ as $t \to \infty$. Similarly, GF $g(s)$ of the form $[g_\delta]$ generates
    the law of immigrants arrival, having the moment of $\delta$-order. If $g'(1-)<\infty$ then
    $[g_\delta]$ holds with $\delta =1$ and $\mathcal{\ell}(t) \to{g'(1-)}$ as $t \to \infty$.

    Throughout the paper $[f_{\nu}]$ and $[g_{\delta}]$ are our Basic assumptions.

    By perforce we allow to forcedly put forward an additional requirement for $\mathcal{L}(x)$ and
    ${\ell}(x)$. So we can write
\begin{equation*}
    {{\mathcal{L}\left( {\lambda x} \right)} \over {\mathcal{L}(x)}}
    = 1 + {\mathcal{O}}\bigl(g(x)\bigr)
    \quad \parbox{2.2cm}{{as} {} $x  \to \infty$.}        \eqno[\textsf {$\mathcal{L}_{\nu}$}]
\end{equation*}
    for each $\lambda > 0$, where $g(x)$ is known positive
    decreasing function so that $g(x) \to 0$ as $x \to \infty $. In this case $\mathcal{L}(x)$ is called
    ${\textbf{SV}}_\infty$ with remainder ${\mathcal{O}}\bigl(g(x)\bigr)$; see {\cite[p.~185, condition SR1]{Bingham}}.
    Wherever we exploit the condition $\left[\mathcal{L}_{\nu} \right]$ we will suppose that
\begin{equation*}
    g(x) = {\mathcal{O}}\left( {{{\mathcal{L}\left(x\right)} \over {x^\nu }}} \right)
    \quad \parbox{2.2cm}{{as} {} $x  \to \infty$.}
\end{equation*}
    Similarly, we also allow a condition
\begin{equation*}
    {{{\ell}\left( {\lambda x} \right)} \over {{\ell}(x)}}
    = 1 + {\mathcal{O}}\bigl(h(x)\bigr)
    \quad \parbox{2.2cm}{{as} {} $x  \to \infty$}     \eqno[\textsf {${\ell}_{\delta}$}]
\end{equation*}
    for each $\lambda > 0$, where
\begin{equation*}
    h(x) = {\mathcal{O}}\left( {{{{\ell}\left(x\right)} \over {x^\delta}}} \right)
    \quad \parbox{2.2cm}{{as} {} $x  \to \infty$.}
\end{equation*}

\subsection {Main Results}                  \label{MySubsec:1.2}

    Since $F (t;s)\to 1$ as $t\to \infty$ uniformly in $s\in[0,1)$ (see Lemma~\ref{MyLem:1} below),
    suffice it to consider ${\mathcal{P}}(t;s) := {\mathcal{P}}_{0}(t;s)$.
    Then under assumptions $[f_\nu]$ and $[g_\delta]$, due to the backward Kolmogorov equation
    ${\partial{F}}/{\partial{t}}=f\left(F\right)$, from \eqref{1.1} formally follows
\begin{equation*}
    {\mathcal{P}}(t;s) = \exp \left\{{\int\limits_s^{F(t;s)}
    {{{g(x)} \over {f(x)}}\,dx}} \right\} \longrightarrow {w}(s)
    \qquad \parbox{2.2cm}{{as} {} $t  \to \infty$,}
\end{equation*}
    where in view of $[f_\nu]$ and $[g_\delta]$,
\begin{equation}                    \label{1.2}
    {w}(s)= \exp \left\{-{\int\limits_s^1 {{\left(1-x\right)^{\gamma -1}
    \textsf{\emph{L}}{\left({1} \over {1-x}\right)}dx}}} \right\}
\end{equation}
    herein $\gamma = \delta - \nu$ and
\begin{equation*}
    \textsf{\emph{L}}(x) :={{\ell(x)} \over {\mathcal{L}(x)}} \,\raise0.8pt\hbox{.}
\end{equation*}

    All appearances, the three cases can be divided concerning the classification of $\mathcal{S}$,
    depending on a sign of $\gamma$. Evidently, integral in \eqref{1.2} converges if
    $\gamma >0$, and diverges if $\gamma <0$. Thus and so, as it was shown in {{\cite{LiPakes2012}}, that
    $\mathcal{S}$ is positive-recurrent if $\gamma > 0$, and it is transient if $\gamma < 0$. The
    special case $\gamma = 0$ implies that $g(s)=f'(s)$ and that $\textsf{\emph{L}}(t) \to 1+\nu$
    as $t \to \infty$. And we get another population process called \textit{Markov Q-process}
    instead of MBPI. We refer the reader to {\cite{Imomov17}} and {\cite{Imomov12}} for the
    details on the Markov Q-process; see also {\cite[pp.~56--58]{ANey}} and {\cite{Pakes99}}
    for the discrete-time case.

    We can see that in the case $\gamma >0$ the function ${\pi}(s)$ generates an invariant measure for MBPI.
    In fact, owing to the functional equation $F(t+\tau;s) = {F\bigl(t; {F(\tau;s)}\bigr)}$ it follows
\begin{eqnarray*}
    {\mathcal{P}}(t+\tau;s)
    & = & \exp \left\{ {\int\limits_0^{t+\tau} {g\left( {F(u;s)} \right)du}} \right\}
    = {\mathcal{P}}(\tau;s)\cdot\exp \left\{ {\int\limits_{\tau}^{t+\tau} {g\left( {F(u;s)} \right)du}} \right\} \\
\nonumber\\
    & = & {\mathcal{P}}(\tau;s) \cdot \exp \left\{ {\int\limits_{0}^{t} {g\left( {F\bigl(u; {F(\tau;s)}\bigr)} \right)du}} \right\}
    =  {\mathcal{P}}(\tau;s)\cdot {\mathcal{P}}\bigl(t; {F(\tau;s)}\bigr)
\end{eqnarray*}
    and taking limit as $t \to \infty$ we have the following Schr\"{o}der type functional equation:
\begin{equation}                    \label{1.3}
    {w}\bigl(F(\tau;s)\bigr) = {\frac{1}{{\mathcal{P}}(\tau;s)}} \, {w}(s)
    \qquad \parbox{2.8cm}{{for any} {} $\tau \in {\mathcal{T}}$.}
\end{equation}
    Writing the power series expansion ${w}(s)= \sum\nolimits_{j \in{\mathcal{S}}}{{w}_j s^j}$,
    the equation \eqref{1.3} implies that ${w}_j = \sum\nolimits_{i \in {\mathcal{S}}} {{w}_i p_{ij}(\tau)}$.

    First of all we observe asymptotic properties of ${p_{00}(t)}$. Our first theorem shows that
    $\ln{p_{00}(t)}$ is asymptotical ${\textbf{SV}}_\infty$. Henceforth further we use a designation
\begin{equation*}
    \tau{(t)} := {(\nu{t})^{1/\nu} \over {\mathcal{N}(t)}} \, \raise 0.8pt\hbox{,}
\end{equation*}
    where the function ${\mathcal{N}}(x)$ is ${\textbf{SV}}_\infty$ defined in Lemma~\ref{MyLem:1} below.

\begin{theorem}                    \label{MyTh:1}
    Let $\gamma > 0$. If assumptions $[\mathcal{L}_{\nu}]$ and $[{\ell}_{\delta}]$ hold, then
\begin{equation}                    \label{1.4}
    -\ln{p_{00}(t)} = {{\,1\,}\over{\gamma}}{\textsf{L}{\bigl(\tau{(t)}\bigr)}} \bigl(1+\kappa(t)\bigr),
\end{equation}
    where
\begin{itemize}
\item  [{(i)}]   if $\delta > 2\nu$, then
    $\kappa(t)=\mathcal{O}\left( 1 \bigl/ t \bigr. \right)$ as $t \to \infty$;

\item [{(ii)}]   if $\delta < 2\nu$, then
    $\kappa(t)=\mathcal{O}\left( \mathcal{N}^{\gamma}(t) \bigl/ t^{\gamma/\nu} \bigr. \right)$ as $t \to \infty$.
\end{itemize}
\end{theorem}

    Another property comes out when $\gamma < 0$.

\begin{theorem}                    \label{MyTh:2}
    Let $\gamma < 0$. If assumptions $[\mathcal{L}_{\nu}]$ and $[{\ell}_{\delta}]$ hold, then
\begin{equation}                    \label{1.5}
    -{{\bigl(\tau{(t)}\bigr)^{-|\gamma|}} \ln{p_{00}(t)}} = {{\,1\,}\over{|\gamma|}}
    {\textsf{L}{\bigl(\tau{(t)}\bigr)}} \bigl(1+\kappa(t)\bigr),
\end{equation}
    where
\begin{itemize}
\item  [{(i)}]   if ${\nu}({\nu}-{\delta}) > \delta$, then
    $\kappa(t)=\mathcal{O}\bigl( {\ell_{\tau}(t)} \bigl/ {t^{\delta/\nu}} \bigr. \bigr)$
    as $t \to \infty$, where ${\ell_{\tau}(t)}$ is ${\textbf{SV}}_\infty$;

\item [{(ii)}]   if ${\nu}({\nu}-{\delta}) < \delta$, then
    $\kappa(t)=\mathcal{O}\left( \mathcal{N}^{|\gamma|}(t) \bigl/ t^{|\gamma|/\nu} \bigr. \right)$ as $t \to \infty$.
\end{itemize}
\end{theorem}

    Now more generally, taking into account Basic assumptions $[f_{\nu}]$ and $[g_{\delta}]$, it would be reasonable
    to seek for an explicit form of the GF of invariant measures depending on the sign of the parameter $\gamma$.

\begin{theorem}                     \label{MyTh:3}
    Let $\gamma > 0$ and assumptions $[\mathcal{L}_{\nu}]$ and $[{\ell}_{\delta}]$ hold. Then
\begin{equation}                    \label{1.6}
    {w}(s) = \exp \left\{ -{{\,1\,}\over{\gamma}} { {g(s)} \over {\Lambda(1-s)} }
    \left( 1+ \mathcal{O}\bigl(\Lambda(1-s)\bigr) \right) \right\}
    \quad \parbox{2.2cm}{{as} {} $s  \uparrow 1$,}
\end{equation}
    where $\Lambda(y)=y^{\nu}\mathcal{L}\left( 1/y\right)$. GF ${w}(s)$ generates
    an invariant distribution with respect to transition probabilities $\{p_{ij}(t)\}$.
\end{theorem}

    In the case $\gamma < 0$, the asymptotic formula \eqref{1.5} suggests that we should look for a limit
    as $t \to \infty$ of the function $e^{T(t)}{\mathcal{P}}(t;s)$ with $T(t)={\bigl(\tau(t)\bigr)^{|\gamma|}}$.
    First wee need to discuss a ${\textbf{SV}}_\infty$ property of $\textsf{\emph{L}}(t)$. In accordance with
    Slowly varying theory, ${\ell}(\cdot)$ and $\mathcal{L}(\cdot)$ are positive and monotone. Moreover,
    by virtue of {\cite[p.~186, Corollary~3.12.3]{Bingham}}, we see
\begin{itemize}
\item[$\diamondsuit$]    $[\mathcal{L}_{\nu}]$ \quad $\Longleftrightarrow$ \quad ${\mathcal{L}(x)}
                        =C_1\left[ 1+ {\mathcal{O}}\bigl(g(x)\bigr) \right]
        \quad \parbox{2.2cm}{{as} {} $t \to \infty$,}$

\item[$\diamondsuit$]    $[{\ell}_{\delta}]$ \quad \, $\Longleftrightarrow$ \quad ${{\ell}(x)}
                        =C_2\left[ 1+ {\mathcal{O}}\bigl(h(x)\bigr) \right]
        \quad \parbox{2.2cm}{{as} {} $t \to \infty$,}$
\end{itemize}
    where $C_1, C_2$ -- positive constants and functions $g(x), h(x)$ are in $[\mathcal{L}_{\nu}]$ and $[{\ell}_{\delta}]$.
    We then can reveal the fact that $\textsf{\emph{L}}(t) \to constant$ as $t \to \infty$, more precisely
\begin{equation*}
    \textsf{\emph{L}}(t) = {{\ell(t)} \over {\mathcal{L}(t)}} =
    {C}\left[ 1+ \mathcal{O}\left({{{\ell}(t)} \over {t^{\delta}}}\right) \right]
    \quad \parbox{2.2cm}{{as} {} $t \to \infty$,}       \eqno[\textsf{${\textsf{L}}_{\gamma}$}]
\end{equation*}
    where $C={{C_1}/{C_2}}$. Especially, we reach an ``excellent result'' in this issue,
    if $C=|\gamma|$. So with  respect to Theorem~\ref{MyTh:2} we obtain the following theorem.

\begin{theorem}                     \label{MyTh:4}
    Let $\gamma < 0$ and $C=|\gamma|$ in $[\textsf{L}_{\gamma}]$. If $1<{{\nu}/{\delta}}<2$, then
\begin{equation}                    \label{1.7}
    {e^{T(t)}} {\mathcal{P}} (t;s) = {\pi}(s) \bigl(1+\rho(t;s)\bigr),
\end{equation}
    where $\rho(t;s)\to 0$ uniformly in $s\in [0, 1)$ as $t \to \infty$ and the
    limiting GF ${\pi}(s) = \sum\nolimits_{j \in{\mathcal{S}}}{\pi_j s^j}$  has the form of
\begin{equation}                    \label{1.8}
    \pi(s) = \exp \left\{{{\,1\,}\over{(1-s)^{|\gamma|}}} + {\int_s^1 {\left[ {{{g(u)} \over {f(u)}}
    + {|\gamma| \over {(1-u)^{1+|\gamma|}}}} \right]du}} \right\}
\end{equation}
    and the set of non-negative numbers $\left\{ {\pi}_j\right\}$ is an invariant measure for $X(t)$.
\end{theorem}

    We notice, the statement of Theorem~\ref{MyTh:4} is compatible with the results of the papers {\cite{PakesSankh}}
    and {\cite{ImomovSFU14}} established for the case of of $\max \bigl\{ f''(1-), g'(1-) \bigr\} <\infty$.
    Thus this theorem essentially strengthens last-mentioned results. With that, in conditions of the
    Theorem~\ref{MyTh:4} right-hand sides of \eqref{1.4} and \eqref{1.5} tend to $1$.

    It is easy to see that under conditions of Theorem~\ref{MyTh:4} the function
\begin{equation}                    \label{1.9}
    {\mathcal{B}}(s): = \exp \left\{ {\int_s^1 {\left[ {{{g(u)} \over {f(u)}}
    + {|\gamma| \over {(1-u)^{1+|\gamma|}}}} \right]du}} \right\}
\end{equation}
    is bounded for $s \in [0, 1)$. So that
\begin{equation*}
    \pi(s) \sim \exp \left\{{{\,1\,}\over{(1-s)^{|\gamma|}}} \right\}
    \quad \parbox{2.0cm}{{as} {} $s \uparrow 1$.}
\end{equation*}

\begin{corollary}                        \label{MyCor:1}
    Under the conditions of Theorem~\ref{MyTh:4}
\begin{equation*}
    {e^{T(t)}}p_{00} (t) = {\mathcal{B}(0)}  \bigl(1+\rho(t)\bigr),
\end{equation*}
    where $\rho(t)\to 0$ as $t \to \infty$ and the function ${\mathcal{B}}(s)$ is defined in \eqref{1.9}.
\end{corollary}

\section{Auxiliaries}   \label{MySec:2}

    Below-mentioned statements have auxiliary character and they will be
    essentially used in proofs of the Main results of the present note.

    At first we recall the following Basic lemma of the theory
    critical Markov branching processes with infinite second moment.

\begin{lemma} [{\cite{ImomovMey20}}, {\cite{Imomov17}}]               \label{MyLem:1}
    If the condition $[f_\nu]$ holds then
\begin{equation}                     \label{2.1}
    R(t;s) = {{{\mathcal N}(t)} \over {(\nu t)^{{1 \mathord{\left/
    {\vphantom {1 \nu }} \right. \kern-\nulldelimiterspace} \nu }} }}
    \cdot \left[ {1 - {{M(t;s)} \over {\nu{t}}}} \right]
\end{equation}
    for all $s\in [0, 1)$, where
\begin{equation}                     \label{2.2}
    {\mathcal N}^{\,\nu}(t) \cdot \mathcal{L}\left({{{\bigl(\nu t \bigr)^{{1
    \mathord{\left/ {\vphantom {1 \nu }} \right. \kern-\nulldelimiterspace}
    \nu }} } \over  {{\mathcal N}(t)}}} \right) \longrightarrow 1
    \quad \parbox{2.2cm}{{as} {} $t  \rightarrow \infty$;}
\end{equation}
    herein $M(t;0)=0$ for all $t>0$ and $M(t;s)\rightarrow {\mathcal{M}(s)}$ as
    $t\rightarrow {\infty}$, where $\mathcal{M}(s)$ is GF of invariant measures of MBP and
\begin{equation*}
    \mathcal{M}(s) = \int\limits_1^{{1 \mathord{\left/ {\vphantom {1 {(1 - s)}}}
    \right. \kern-\nulldelimiterspace}{(1 - s)}}}
    {{{dx} \over {x^{1 - \nu }\mathcal{L}(x)}}}\,\raise0.8pt\hbox{.}
\end{equation*}
\end{lemma}

    The following two lemmas describe a role of slowly varying functions with remainder in integration.

\begin{lemma}              \label{MyLem:2}
    Let ${L}(t)$ is ${\textbf{SV}}_\infty$-function with remainder $r(t)$. Then
\begin{itemize}
\item  [{(i)}]   for $\sigma > 0$ and for $0 < c <t$
\begin{equation}                    \label{2.3}
    {\int\limits_{c}^t {{y^{-(1+\sigma)}{L}(y)dy}}}  = {{\,1\,} \over {\sigma}}
    {1 \over {\,c^{\sigma}}} {L}(t) \bigl({1 - {\mu}^{\sigma}}\bigr) \bigl({1+ r(t)}\bigr)
    \quad \parbox{2.2cm}{{as} {} $t \to \infty$,}
\end{equation}
    where ${\mu}:={c}/t$;

\item [{(ii)}]   and for $\sigma < 0$
\begin{equation}                    \label{2.4}
    {\int\limits_{c}^t {{y^{-(1+\sigma)}{L}(y)dy}}}  = {{\,1\,} \over {|\sigma|}}
    {L}(t) {t^{|\sigma|}} \bigl({1 - {\mu}^{|\sigma|}}\bigr) \bigl({1+ r(t)}\bigr)
    \quad \parbox{2.2cm}{{as} {} $t \to \infty$.}
\end{equation}
\end{itemize}
\end{lemma}

\begin{proof}
    We write
\begin{eqnarray}                    \label{2.5}
    \mathcal{I}(t)
    & := & {\int\limits_{c}^t {{y^{-(1+\sigma)}{L}(y)dy}}}         \nonumber\\
\nonumber\\
    & = & {{L(t)} \over {t^{\sigma}}}\left[\int_{\mu}^{1}{y^{-(1+\sigma)}dy} +
    \int_{0}^{1}{\left[{{L(yt)} \over {L(t)}}-1 \right] {\textbf{I}_{\textsf{M}}(y)} y^{-(1+\sigma)}dy}\right],
\end{eqnarray}
    where $\textbf{I}_{\textsf{M}}(x)$ is an Indicator function of $\textsf{M}=[{\mu}, 1]$. By Potter's
    Theorem~{\cite[p.~25]{Bingham}}, slowly varying part ${{L(yt)}/{L(t)}}$ in last integrand on the
    right-hand side is bounded and tends to $1$ as $t \to \infty$ uniformly in $0<y\leq{1}$.
    Thus, since $L(\cdot)$ is ${\textbf{SV}}_\infty$ with remainder $r(\cdot)$ and
\begin{equation*}
    \int_{\mu}^{1}{y^{-(1+\sigma)}dy} = {{\,1\,}\over{\sigma}}  \left( {1\over{\mu^{\sigma}}} - 1 \right),
\end{equation*}
    we have
\begin{equation}                    \label{2.6}
    \mathcal{I}(t)= {{\,1\,}\over{\sigma}} {{L(t)} \over {t^{\sigma}}}
    \left( {1\over{\mu^{\sigma}}} - 1 \right) \bigl({1+ r(t)}\bigr).
\end{equation}
    Now \eqref{2.3} and \eqref{2.4} easily follow from \eqref{2.5} and \eqref{2.6}.
\end{proof}

\begin{lemma}              \label{MyLem:3}
    Let ${L}(t)$ is ${\textbf{SV}}_\infty$ with remainder $r(t)$. Then for $\sigma > 0$
\begin{equation}                    \label{2.7}
    {\int\limits_{t}^{\infty} {{y^{-(1+\sigma)}{L}(y)dy}}}  = {{\,{L}(t)\,} \over {\sigma}}
    {1 \over {\,t^{\sigma}}} \bigl({1+ r(t)}\bigr)
    \quad \parbox{2.2cm}{{as} {} $t \to \infty$.}
\end{equation}
\end{lemma}

\begin{proof}
    Considering $\int_{1}^{\infty}{u^{-(1+\sigma)}du}=1/{\sigma}$, we write
\begin{equation}                    \label{2.8}
    {\int\limits_{t}^{\infty} {{y^{-(1+\sigma)}{L}(y)dy}}}
    = {{\,1\,}\over{\sigma}} {{L(t)} \over {t^{\sigma}}} \left[1 + {\sigma}
    \int_{1}^{\infty}{\left[{{L(yt)} \over {L(t)}}-1 \right] y^{-(1+\sigma)}dy}\right].
\end{equation}
    As in proof of Lemma~\ref{MyLem:2}, second term in brackets on right-hand side of \eqref{2.8}
    tends to $0$ uniformly in $y>{1}$ with speed rate $r(t)$. Thus we have \eqref{2.7}.
\end{proof}

\section{Proof of Main Results}   \label{MySec:3}

    In this final section we consistently prove the Main results.

\begin{proof} [Proof of Theorem~\ref{MyTh:1}]
    In conditions of theorem
\begin{eqnarray*}
    {p}_{00} (t)
    & = & \exp \left\{ {\int\limits_0^t {g\left({F(u;s)} \right)du}} \right\}  \\
\nonumber\\
    & = & \exp \left\{-{\int\limits_0^{F(t)} {{\left(1-u\right)^{\gamma -1}
    \textsf{\emph{L}}{\left({1} \over {1-u}\right)}du}}} \right\}
    = \exp \left\{-{\int\limits_{1}^{1/R(t)} {{y^{-(1+\gamma)}
    \textsf{\emph{L}}{\left(y \right)}dy}}} \right\},
\end{eqnarray*}
    where $R(t)=R(t;0)$. Now in last integral we will use Lemma~\ref{MyLem:2}. Then
\begin{equation}                    \label{3.1}
    -\ln{p_{00}(t)} = {{\,1\,}\over{\gamma}}{\textsf{\emph{L}}{\left(\tau{(t)}\right)}}
    \left(1- {{1}\over {\bigl(\tau(t)\bigr)^{\gamma}}} \right) \left(1+ r\bigl(\tau(t)\bigr)\right)
    \quad \parbox{2.2cm}{{as} {} $t  \to \infty$,}
\end{equation}
    since ${\tau}(t)={1 \bigl/ R(t) \bigr.}$. In considering case
    we can make sure that the remainder for ${\textbf{SV}}_\infty$-function $\textsf{\emph{L}}(t)$ is
\begin{equation*}
    r(t) = \mathcal{O}\left({{\mathcal{L}(t)} \over {t^{\nu}}}\right)
\end{equation*}
    and owing to the property \eqref{2.2}, $r\bigl(\tau(t)\bigr) = \mathcal{O}\left( 1 \bigl/ t \bigr. \right)$.
    Therefore the decreasing speed to $1$ of the product of two last terms in brackets in
    right-hand side of \eqref{3.1} depends on ${\gamma}>{\nu}$ or ${\gamma}<{\nu}$.
    Thus, we obtain tail-part form $\kappa(t)$ in \eqref{1.4}.

    The Theorem is proved.
\end{proof}

\begin{proof} [Proof of Theorem~\ref{MyTh:2}]
    Similarly as in the proof of Theorem~\ref{MyTh:1}, using Lemma~\ref{MyLem:2}, we write
\begin{equation}                    \label{3.2}
    -\ln{p_{00}(t)} = {{\,1\,}\over{|\gamma|}} {\textsf{\emph{L}}{\left(\tau{(t)}\right)}} \bigl(\tau(t)\bigr)^{|\gamma|}
    \left(1-{{1}\over {\bigl(\tau(t)\bigr)^{|\gamma|}}} \right) \left(1+ r\bigl(\tau(t)\bigr)\right)
    \quad \parbox{2.2cm}{{as} {} $t  \to \infty$.}
\end{equation}
    Hereof we easily reach to \eqref{1.5}. To get the tail-part form $\kappa(t)$ we first make
    sure that the remainder for ${\textbf{SV}}_\infty$-function $\textsf{\emph{L}}(t)$ is
\begin{equation*}
    r(t) = \mathcal{O}\left({{{\ell}(t)} \over {t^{\delta}}}\right)
\end{equation*}
    (see also, $[\textsf{L}_{\gamma}]$) and we see $r\bigl(\tau(t)\bigr) = \mathcal{O}\left({{{\ell}_{\tau}(t)}
    \bigl/ {t^{\delta/\nu}} \bigr.}\right)$, where ${\ell}_{\tau}(t)$ is ${\textbf{SV}}_\infty$. So the
    decreasing speed to $1$ of the product of two expressions in brackets in right-hand side of \eqref{3.2}
    depends on ${|\gamma|}>{\delta/\nu}$ or ${|\gamma|}<{\delta/\nu}$. Thus, we obtain forms of $\kappa(t)$.

    The Theorem is proved.
\end{proof}

\begin{proof} [Proof of Theorem~\ref{MyTh:3}]
    Substituting $y:=(1-x)^{-1}$, we rewrite \eqref{1.2} as follows:
\begin{equation}                    \label{3.3}
    {w}(s)= \exp \left\{-{\int\limits_{1/{(1-s)}}^{\infty} {{y^{-(1+\gamma)}
    \textsf{\emph{L}}{\left(y \right)}dy}}} \right\}.
\end{equation}
    Now we can use Lemma~\ref{MyLem:3} in last integral. Then
\begin{equation}                    \label{3.4}
    {\int\limits_{1/{(1-s)}}^{\infty} {{y^{-(1+\gamma)}
    \textsf{\emph{L}}{\left(y \right)}dy}}} = {{\,1\,}\over{\gamma}} {{\left(1-s\right)^{\gamma}
    \textsf{\emph{L}}{\left({1} \over {1-s}\right)}} } \left( 1+ r{\left({1} \over {1-s}\right)} \right)
    \quad \parbox{2.0cm}{{as} {} $s  \uparrow 1$.}
\end{equation}
    Since the remainder of $\textsf{\emph{L}}(t)$ is $r(t) = \mathcal{O}\left({{\mathcal{L}(t)}
    \bigl/ {t^{\nu}}} \bigr. \right)$ for $\gamma >0$ then considering our designations, formula~\eqref{1.6}
    readily follows from \eqref{3.3} and \eqref{3.4}. Undoubtedly, in our case, functional equation \eqref{1.3} is satisfied.
    Writing now the power series expansion ${w}(s)= \sum\nolimits_{j \in{\mathcal{S}}}{{w}_j s^j}$,
    it implies an invariant property ${w}_j = \sum\nolimits_{i \in {\mathcal{S}}} {{w}_i p_{ij}(\tau)}$.

    To complete the proof, it suffices to verify that $w(1-)=1$.
\end{proof}

\begin{proof} [Proof of Theorem~\ref{MyTh:4}]
    We write, as before,
\begin{eqnarray}                    \label{3.5}
    {e^{T(t)}} {\mathcal{P}} (t;s)
    & = & \exp \left\{ {\bigl(\tau(t)\bigr)^{|\gamma|}} + {\int\limits_0^{t} {g\left( {F(u;s)} \right)du}} \right\}   \nonumber \\
    \nonumber  \\
    & = & \exp \left\{ \left[ {\bigl(\tau(t)\bigr)^{|\gamma|}}-{\bigl(\tau(t;s)\bigr)^{|\gamma|}} \right]
    + {\bigl(\tau(t;s)\bigr)^{|\gamma|}} + {\int\limits_s^{F(t;s)} {{{g(x)} \over {f(x)}}\,dx}} \right\},
\end{eqnarray}
    where ${\tau(t;s)}={R^{-1}(t;s)}$. Since ${\tau(t)}={\tau(t;0)}$, using \eqref{2.1} we obtain
\begin{equation*}
    {\bigl(\tau(t)\bigr)^{|\gamma|}}-{\bigl(\tau(t;s)\bigr)^{|\gamma|}} \sim
    - |\gamma| \, {\bigl(\tau(t)\bigr)^{|\gamma|}} \, {{M(t;s)} \over {{\nu}t}}
    \quad \parbox{2.2cm}{{as} {} $t  \to \infty$.}
\end{equation*}
    It follows from assertion of Lemma~\ref{MyLem:1}, the function $M(t;s)$ is bounded in $s\in [0, 1)$. Hence
\begin{equation}                    \label{3.6}
    {\bigl(\tau(t)\bigr)^{|\gamma|}}-{\bigl(\tau(t;s)\bigr)^{|\gamma|}}
    = \mathcal{O}\left({{\mathcal{L}_{\gamma}(t)} \over {t^{\delta/\nu} }}\right)
    \quad \parbox{2.2cm}{{as} {} $t  \to \infty$,}
\end{equation}
    where ${\mathcal{L}_{\gamma}(t)} {\mathcal{N}^{-|\gamma|}(t)} \to 1$
    as $t \to \infty$. Along with this we can easily make sure that
\begin{equation}                    \label{3.7}
    {\bigl(\tau(t;s)\bigr)^{|\gamma|}} = {{\,1\,}\over{(1-s)^{|\gamma|}}}
    + {\int\limits_s^{F(t;s)} {{{|\gamma| \over {(1-u)^{1+|\gamma|}}}} du}}.
\end{equation}
    Combining statements \eqref{3.5}--\eqref{3.7}, we obtain
\begin{equation*}
    {e^{T(t)}} {\mathcal{P}} (t;s) = \exp \left\{{{\,1\,}\over{(1-s)^{|\gamma|}}}
    + {\int\limits_s^{F(t;s)} {\left[ {{{g(u)} \over {f(u)}}
    + {|\gamma| \over {(1-u)^{1+|\gamma|}}}} \right]du}}
    + \mathcal{O}\left({{\mathcal{L}_{\gamma}(t)} \over {t^{\delta/\nu} }}\right) \right\}.
\end{equation*}
    Now we reach to \eqref{1.7} with expression of $\pi(s)$ in the form~\eqref{1.8}, taking limit as
    $t \to \infty$ in last one. Finally, we can verify that the function $\pi(s)$ satisfies
    the equation~\eqref{1.3}. Thus it generates an invariant measure for $X(t)$.
\end{proof}

\begin{proof} [Proof of Corollary~\ref{MyCor:1}]
    The statement is immediately obtained from relation~\eqref{1.7} setting $x = 0$ there.
\end{proof}


\medskip

\begin{thebibliography}{99}

\small

\bibitem{ANey} K.~B.~Athreya, P.~E.~Ney, {\it Branching processes}, Springer, New York, 1972.

\bibitem{Bingham} N.~H.~Bingham, C.~M.~Goldie and J.~L.~Teugels,  {\it Regular Variation}, Cambridge University Press, 1987.

\bibitem{ImomovMey20} A.~A.~Imomov, A.~Kh.~Meyliyev, {\it On application of slowly varying functions with remainder in the theory of Markov
                Branching Processes with mean one and infinite variance}, {Ukrainian Math. Journal}, \textbf{in Press}.

\bibitem{Imomov17} A.~A.~Imomov,  {\it On Conditioned Limit Structure of the Markov Branching Process without Finite Second Moment},
                {Malaysian Journal of Mathematical Sciences}, \textbf{11}:3 (2017), 393--422.

\bibitem{ImomovSFU14} A.~A.~Imomov, {\it On Long-Term Behavior of Continuous-Time Markov Branching Processes Allowing Immigration},
                {Journal of Siberian Federal University: Mathematics and Physics},  \textbf{7}:4 (2014), 429--440.

\bibitem{Imomov12} A.~A.~Imomov, {\it On Markov analogue of Q-processes with continuous time},
                {Theory of Probability and Mathematical Statistics}, \textbf{84} (2012), 57--64.

\bibitem{LiPakes2012} J.~Li, A.~Chen, A.~G.~Pakes, {\it Asymptotic properties of the Markov Branching Process with Immigration},
                {Journal of Theoretical Probability}, \textbf{25} (2012),  122–-143.

\bibitem{Pakes99} A.~G.~Pakes, {\it Revisiting conditional limit theorems for the mortal simple branching process},
                {Bernoulli}, \textbf{5}:6 (1999), 969--998.

\bibitem{PakesSankh} A.~G.~Pakes,
                {\it On Markov branching processes with immigration},
                {Sankhy\=a: The Indian Journal of Statistics}, \textbf{A37} (1975), 129--138.

\bibitem{SenetaRV} E.~Seneta, {\it Regularly Varying Functions}, Springer, Berlin, 1976.

\bibitem{Sevast57} B.~A.~Sevastyanov, {\it Limit theorems for stochastic branching processes of special form},
                {Theory of Probabability and its Appl.}, \textbf{2}:3 (1957), 360--374 (Russian).

\bibitem{Sevast71} B.~A.~Sevastyanov, {\it Branching processes}, Nauka, Moscow, 1971 (Russian).

\end{thebibliography}
\end{document}